\documentclass[smallextended]{svjour3}       % onecolumn (second format)

\smartqed  % flush right qed marks, e.g. at end of proof
\usepackage{graphicx}
\usepackage{multicol}
\usepackage{amsmath}
\usepackage{amssymb}
\usepackage{commath}
\usepackage{mathrsfs}
\usepackage{amsfonts}
\usepackage{graphicx}
\usepackage{subcaption}
\usepackage{comment}
\usepackage{bm}
\usepackage{amsopn}
\usepackage{hyperref}

\def\Re{\mathop{\rm Re}\nolimits}

\newcommand{\ph}{\varphi}

\let\T\intercal

\def\(#1\){{\left(#1\right)}}

\def\hlink#1{\hbox{\hyperlink{#1}{#1}}}
\def\Hlink#1#2{\hbox{\hyperlink{#1}{#2}}}

\def\p2{{p_2}}
\def\e1{{e_1}}

\def\Re{\mathop{\rm Re}\nolimits}

\DeclareMathAlphabet{\mathssbf}{OT1}{cmss}{bx}{n}

\let\ph\varphi

\graphicspath{{fig/}}

\def\hlink#1{\hbox{\hyperlink{#1}{#1}}}

\newtheorem{statement}{Statement}

\ifpdf
\hypersetup{
  pdftitle={Schur Decomposition for Stiff Differential Equations},
  pdfpagemode=UseOutlines,
  citebordercolor=0 0 1,
  colorlinks=true,
  breaklinks=true,
  pdfauthor={Thoma Zoto and John C. Bowman}
  pdfpagetransition=Dissolve,
}
\fi

\journalname{Journal of Scientific Computing}
\begin{document}

\title{Schur Decomposition for Stiff Differential Equations\thanks{
Financial support for this work was provided by grants RES0043585 and
RES0046040 from the Natural Sciences and Engineering Research Council
of Canada.}}


\author{Thoma Zoto \and John C. Bowman}


\institute{T. Zoto \at
  RAPSODI, Inria Lille--Nord Europe,
  Villeneuve d'Ascq 59650, France\\
  \email{thomazoto1@gmail.com}
           \and
           J.C. Bowman\at
           Department of Mathematical and Statistical Sciences,
           University of Alberta, Edmonton, Alberta T6G 2G1, Canada\\
           \email{bowman@ualberta.ca}
}

\date{Submitted: May 13, 2023}
%\date{Received: / Accepted:}

\maketitle

\begin{abstract}
  A quantitative definition of numerical stiffness for initial value problems
is proposed. Exponential integrators can
effectively integrate linearly stiff systems, but they become expensive
when the linear coefficient is a matrix, especially when the time step
is adapted to maintain a prescribed local error.
Schur decomposition is shown to avoid the need for computing matrix
exponentials in such simulations, while still circumventing linear stiffness.
\end{abstract}

\section{Introduction}
The time integration of initial value problems is ubiquitous in
simulations of physical phenomena. Consider a first-order initial
value problem of the form
\begin{equation}
    \label{eq-LF}
    \frac{dy}{dt} = f(t,y(t)) = F(t,y(t)) - L y, \qquad y(0)=y_0,
\end{equation}
where $y$ is a vector, $F$ is an analytic function, and $L$ is a constant
matrix. Numerical approximations of a future estimate $y_{n+1}$
can be obtained using an explicit Runge--Kutta (RK) method:
\begin{equation}
    \label{eq-RK}
    y_n^{i+1} = y_n^0 +h\sum_{j=0}^ia_{ij}f(t_n+c_jh,y_n^j)
                \quad i=0,\dots,s-1,
\end{equation}
where $n=0,1,\ldots$, $y_0^0=y_0$, $y_n^s=y_{n+1}$, $h$ is the {\it time step\/},
$t_n=nh$, $a_{ij}$ are the Runge--Kutta {\it weights\/}, and $c_j$ are the
{\it step fractions\/} for stage $j$. 
For $n \ge 1$, $y_n^0=y_{n-1}^s$ is the approximation of the solution
at time $nh$, also denoted by $y_n$.
%In the above notation, $y_n^s$ is the approximation of the solution
%at time $(n+1)h$, also denoted by $y_{n+1}$.
It is customary to organize the weights in a
{\it Butcher tableau\/} (Table~\ref{t-RK}).
\begin{table}[htbp]
\[
\renewcommand\arraystretch{1.3}
\begin{array}
{c|cccccc}
0\\
c_1
    & a_{00}\\
c_2 
    & a_{10}
    & a_{11}\\
\vdots
    & \vdots
    & \vdots
    & \ddots\\
c_{s-1}
    & a_{(s-2)0}
    & \cdots
    & \cdots
    & a_{(s-2)(s-2)}\\
\hline
1
    & a_{(s-1)0}
    & \cdots
    & \cdots
    & a_{(s-1)(s-1)}\\
\end{array}
\]
    \caption{General Runge--Kutta tableau}
    \label{t-RK}
\end{table}
For some problems, explicit Runge--Kutta methods may require a very
small time step. This failure is often called
numerical stiffness, and is described and defined in
Section~\ref{sec-stiffness}. While one might consider
implicit methods, they require iteration within a time step.
Exponential Runge--Kutta (ERK) integrators provide an alternative to
implicit methods for solving stiff problems.
These are explicit methods that alleviate the burden of stiffness
and have similar structure to explicit Runge--Kutta methods
(a notable difference being that the weights are not constants, but
depend on the matrix $L$):
\begin{equation}
    \label{eq-ERK}
    y_n^{i+1} = e^{-h L}y_n^0 +h\sum_{j=0}^ia_{ij}(-hL) F(t_n+c_jh,y_n^j)
                \quad i=0,\dots,s-1.
\end{equation}
We briefly describe these
methods in Section~\ref{exp-integrators}. In Section~\ref{schur}, we
show how Schur decomposition can be used to improve the efficiency of
exponential integrators when~$L$ is a nondiagonal matrix.
We conclude the paper with some numerical examples and applications in
Section~\ref{applications}.

\section{Stiffness of Explicit Methods}
\label{sec-stiffness}

We focus on solving ordinary dif\-ferential equations (ODEs) or
systems of ODEs of the form \eqref{eq-LF}.
In practical applications, such systems of ODEs arise upon
spatial discretization of a PDE via finite differences, finite elements,
or spectral transforms.

One of the most common ways of quantifying stiffness in the
literature is the concept of the \textit{stiffness ratio}. If we
denote $\lambda_{\min}$ and $\lambda_{\max}$ to be the smallest and the
largest eigenvalues
of $L$ (in modulus), then the stiffness ratio is
\begin{equation}
    \frac{\mid{\Re \lambda_{\max}\mid}}{\mid{\Re \lambda_{\min}\mid}}.
\end{equation}
The larger this ratio is, the more stiff the system is considered to be.
However, many authors have realized that this is not the best
definition for the phenomena because, if $\lambda_{\min}$ is zero, then
the stiffness ratio is infinite, but the problem may not be stiff at
all. Lambert describes a few other attempts to define stiffness based on
stability, accuracy, or decay rates, 
although none of them are satisfactory, either due to the existence of
a counterexample or due to their qualitative rather than quantitative
nature \cite{Lambert91}.
One statement that Lambert seems to accept (and one that has also has
been used consistently in the literature alongside the stiffness
ratio) is:
\begin{statement}
    \label{stiffness-Lambert}
    If a numerical method with a finite region of absolute stability,
    applied to a system with any initial conditions, is forced to use
    in a certain interval of integration a step-length which is
    excessively small in relation to the smoothness of the exact
    solution in that interval, then the system is said to be stiff in
    that interval.
\end{statement}
This is a helpful definition if we want to just know by testing
whether a system is
stiff or not, but it requires actually applying a method and observing
whether it fails.
Typically, stiff systems are solved numerically by first
trying an explicit method and, if that fails for reasonable
step sizes, switching to
an implicit method. Understood in this way, Statement~\ref{stiffness-Lambert} portrays stiffness as a property that
depends on the chosen numerical method and not as an intrinsic
phenomenon of the ODE system itself. We want to define stiffness
so that it depends only on the ODE system and
is helpful for applying exponential integrators, distinguishing between
stiffness coming from a linear term and stiffness coming from a
nonlinear term.

In order to gain some insights about stiffness, let us first explore
how explicit and implicit methods solve ODEs.
We recall the two simplest time integrating schemes:
the explicit Euler method $y_{n+1}=y_n + hf(y_n)$ and implicit Euler
method $y_{n+1}=y_n + hf(y_{n+1})$.
In Figure~\ref{fig-curves-tang}, we present a graphical description of
why the explicit Euler method performs poorly
when $F(t,y(t))=0$, $L=20$ and $y_0=1$ in~\eqref{eq-LF}.
At each time $t_n$, these methods compute an approximation to the
exact solution that lies on a nearby solution curve. 
The implicit Euler method evolves the
solution in the direction of the tangent line
to the nearby solution curve at the point $(t_{n+1},y_{n+1})$
(red segment), while the explicit Euler method evolves in the
direction of the tangent line to the nearby solution curve at the
point $(t_n,y_n)$ (green segment). 
The tangent line at $(t_{n+1},y_{n+1})$ is much 
closer to the direction of the exact solution at the point
$(t_n,y(t_n))$, unlike the tangent line at $(t_n,y_n)$.
The bigger the step size, the more aligned the tangent line at
$(t_{n+1},y_{n+1})$ will be with the direction of the exact solution
at the point $(t_n,y(t_n))$, in contrast to the 
tangent line at $(t_n,y_n)$.
This phenomenon causes the explicit Euler method
to work only for sufficiently small step sizes.
The corresponding slope field is shown in Figure~\ref{fig-field-E}.
\begin{figure}[h]
\begin{subfigure}{0.49\linewidth}
    \centering
    \includegraphics[width=0.9\linewidth]{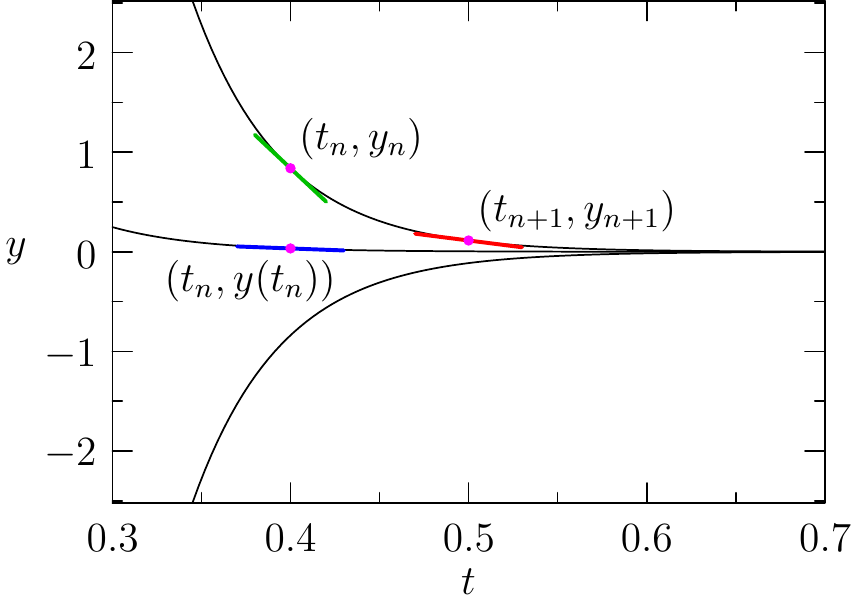}
    \caption{Graphical interpretation of explicit and implicit Euler methods.}
    \label{fig-curves-tang}
\end{subfigure}
\,
\begin{subfigure}{0.49\textwidth}
    \centering
    \includegraphics[width=0.9\linewidth]{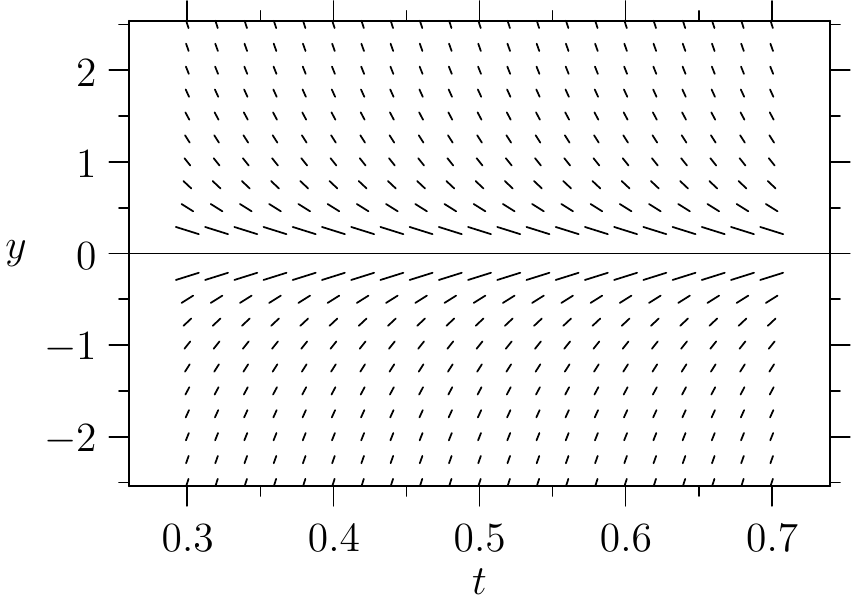}
    \caption{Slope field of $\displaystyle\frac{dy}{dt}=-20y$.}
    \label{fig-field-E}
\end{subfigure}
\end{figure}

This point of view on stiffness can also be extended to general
systems of ODEs. It was first
introduced by \cite{Curtiss52} and described by
\cite{Lambert91}. We will expand on these ideas to make
them applicable to the stiff differential
equations that we aim to solve with exponential integrators.
We borrow the following two systems from \cite{Lambert91}:
\begin{align}
  \hypertarget{System 1}\ 
    \text{System 1:}\nonumber\\
    \begin{bmatrix}
        y_1'\\
        y_2'
    \end{bmatrix}
    &=
    \begin{bmatrix}
        -2 & 1\\
        1 & -2
    \end{bmatrix}
    \begin{bmatrix}
        y_1\\
        y_2
    \end{bmatrix}
    +
    \begin{bmatrix}
        2\sin(t)\\
        2(\cos(t) - \sin(t))
    \end{bmatrix}
    , \quad
    \begin{bmatrix}
        y_1(0)\\
        y_2(0)
    \end{bmatrix}
    =
    \begin{bmatrix}
        2\\
        3
    \end{bmatrix},
\end{align}
\begin{align}
  \hypertarget{System 2}\ 
    \text{System 2:}\nonumber\\
    \begin{bmatrix}
        y_1'\\
        y_2'
    \end{bmatrix}
    &=
    \begin{bmatrix}
        -2 & 1\\
        998 & -999
    \end{bmatrix}
    \begin{bmatrix}
        y_1\\
        y_2
    \end{bmatrix}
    +
    \begin{bmatrix}
        2\sin(t)\\
        999(\cos(t) - \sin(t))
    \end{bmatrix}
    , \quad
    \begin{bmatrix}
        y_1(0)\\
        y_2(0)
    \end{bmatrix}
    =
    \begin{bmatrix}
        2\\
        3
    \end{bmatrix}.
\end{align}
The particular solution for the given initial conditions is the same
for~\hlink{System 1} and~\hlink{System 2}:
\begin{equation}
    \label{soln-23}
    \begin{bmatrix}
        y_1(t)\\
        y_2(t)
    \end{bmatrix}
    = 2\exp(-t)
    \begin{bmatrix}
        1\\
        1
    \end{bmatrix}
    +
    \begin{bmatrix}
        \sin(t)\\
        \cos(t)
    \end{bmatrix}.
\end{equation}
Lambert selected the initial
condition $y(0) = [2,3]^\T$, so that~\hlink{System 1} and
\hlink{System 2} have the same exact solution,
to emphasize that the concept of stiffness does
not depend on the particular solution.
In Figure~\ref{fig-soln-curves}
we plot the components as well as the phase curves
corresponding to various initial conditions
for both systems, allowing one to examine nearby solutions.
\begin{figure}
\centering
\begin{subfigure}{.5\textwidth}
    \centering
    \includegraphics[width=\linewidth,height=0.2\textheight,keepaspectratio]{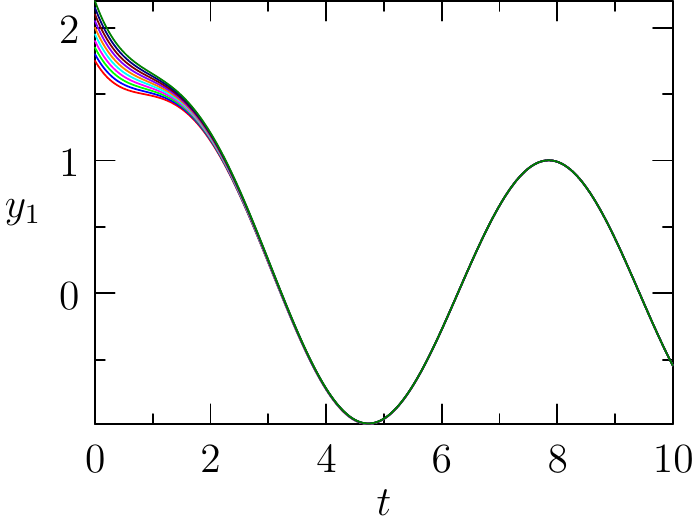}
    \caption{$y_1(t)$ solution curves for \protect\hlink{System 1}.}
    \label{fig-y1-P1}
\end{subfigure}%
\begin{subfigure}{.5\textwidth}
    \centering
    \includegraphics[width=\linewidth,height=0.2\textheight,keepaspectratio]{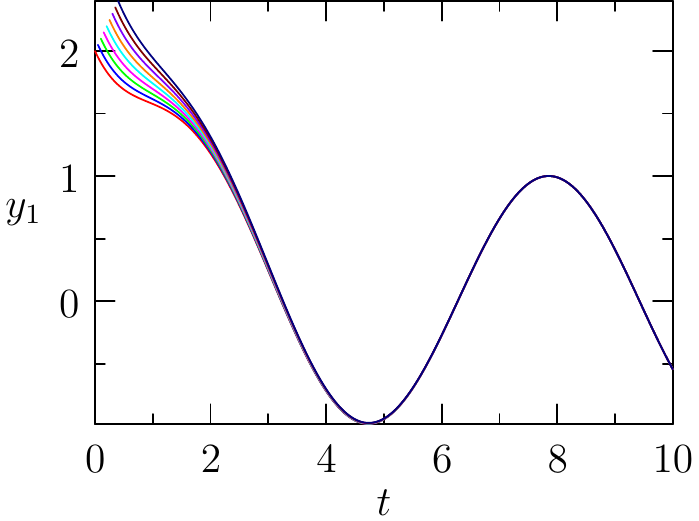}
    \caption{$y_1(t)$ solution curves for \protect\hlink{System 2}.}
    \label{fig-y1-P2}
\end{subfigure}%
\\[\smallskipamount]
\begin{subfigure}{.5\textwidth}
    \centering
    \includegraphics[width=\linewidth,height=0.2\textheight,keepaspectratio]{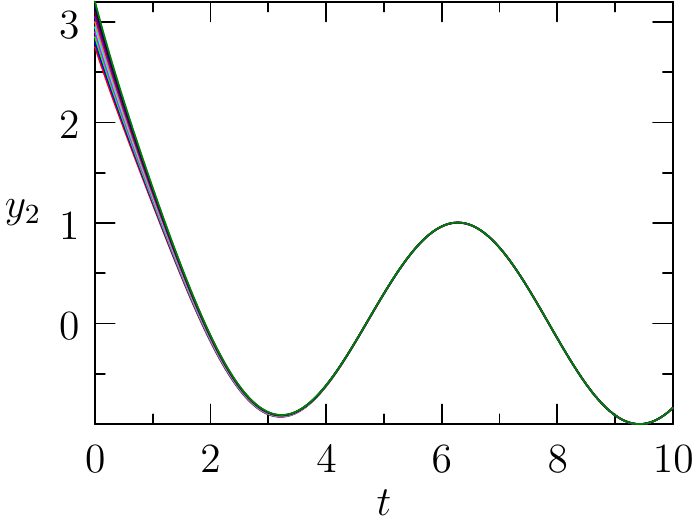}
    \caption{$y_2(t)$ solution curves for \protect\hlink{System 1}.}
    \label{fig-y2-P1}
\end{subfigure}%
\begin{subfigure}{.5\textwidth}
    \centering
    \includegraphics[width=\linewidth,height=0.2\textheight,keepaspectratio]{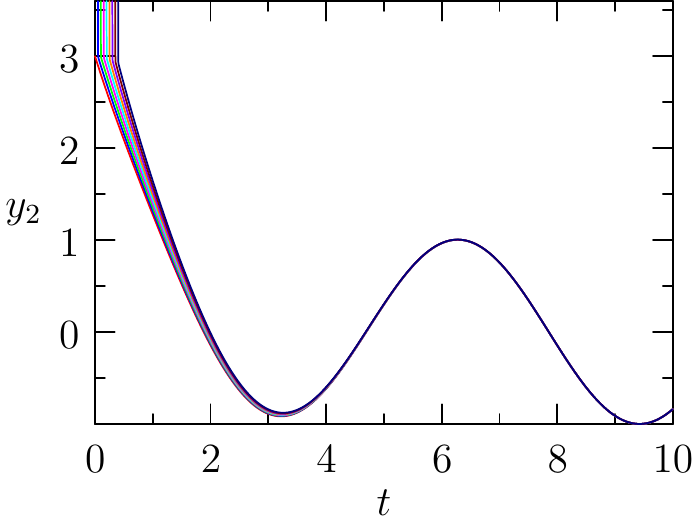}
    \caption{$y_2(t)$ solution curves for \protect\hlink{System 2}.}
    \label{fig-y2-P2}
\end{subfigure}%
\\[\smallskipamount]
\begin{subfigure}{.5\textwidth}
    \centering
    \includegraphics[width=\linewidth,height=0.2\textheight,keepaspectratio]{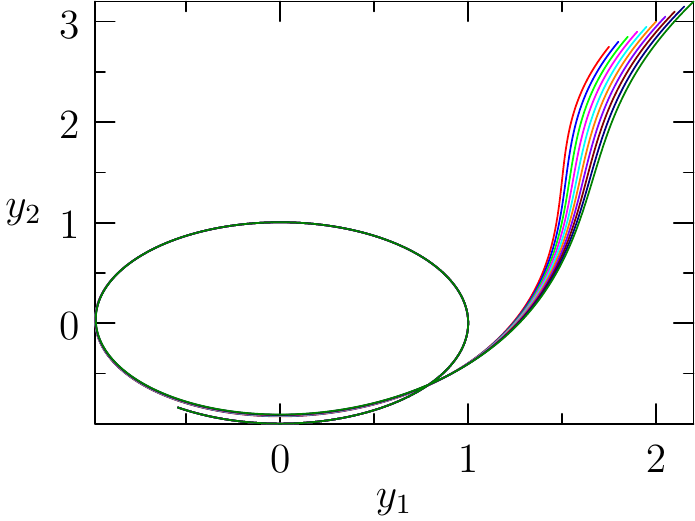}
    \caption{Phase curves for \protect\hlink{System 1}.}
    \label{fig-phase-P1}
\end{subfigure}%
\begin{subfigure}{.5\textwidth}
    \centering
    \includegraphics[width=\linewidth,height=0.2\textheight,keepaspectratio]{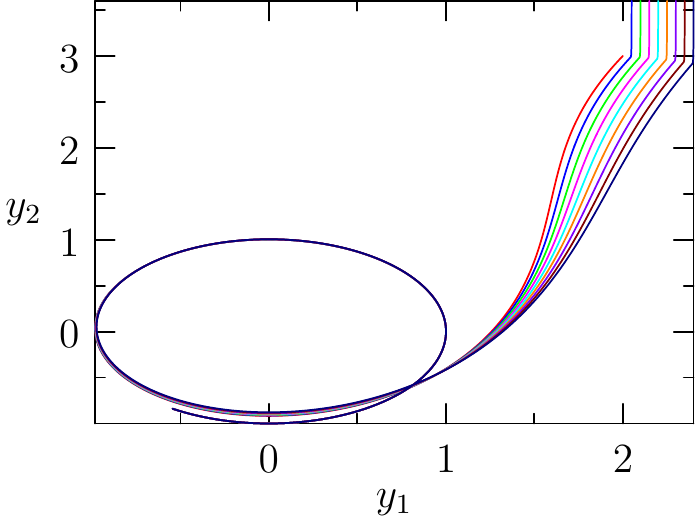}
    \caption{Phase curves for \protect\hlink{System 2}.}
    \label{fig-phase-P2}
\end{subfigure}%
\caption{Plots of nearby solution curves for \protect\hlink{System 1} and \protect\hlink{System 2}.}
\label{fig-soln-curves}
\end{figure}

\begin{figure}
\centering
\begin{subfigure}{.5\textwidth}
    \centering
    \includegraphics[width=\linewidth,height=0.2\textheight,keepaspectratio]{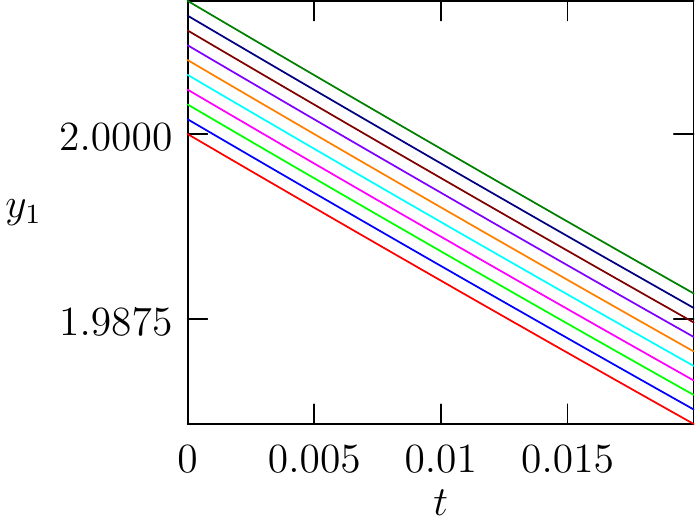}
    \caption{$y_1(t)$ curves for \protect\hlink{System 1} (zoomed).}
    \label{fig-y1-P1-zoom}
\end{subfigure}%
\begin{subfigure}{.5\textwidth}
    \centering
    \includegraphics[width=\linewidth,height=0.2\textheight,keepaspectratio]{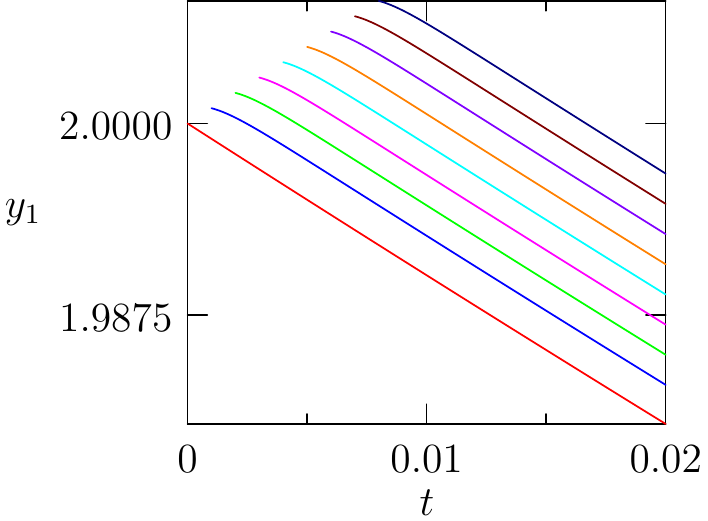}
    \caption{$y_1(t)$ curves for \protect\hlink{System 2} (zoomed).}
    \label{fig-y1-P2-zoom}
\end{subfigure}%
\\[\smallskipamount]
\begin{subfigure}{.5\textwidth}
    \centering
    \includegraphics[width=\linewidth,height=0.2\textheight,keepaspectratio]{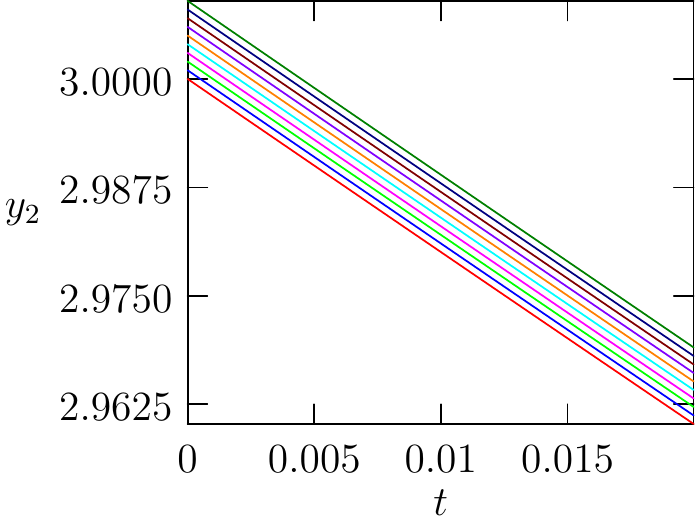}
    \caption{$y_2(t)$ curves for \protect\hlink{System 1} (zoomed).}
    \label{fig-y2-P1-zoom}
\end{subfigure}%
\begin{subfigure}{.5\textwidth}
    \centering
    \includegraphics[width=\linewidth,height=0.2\textheight,keepaspectratio]{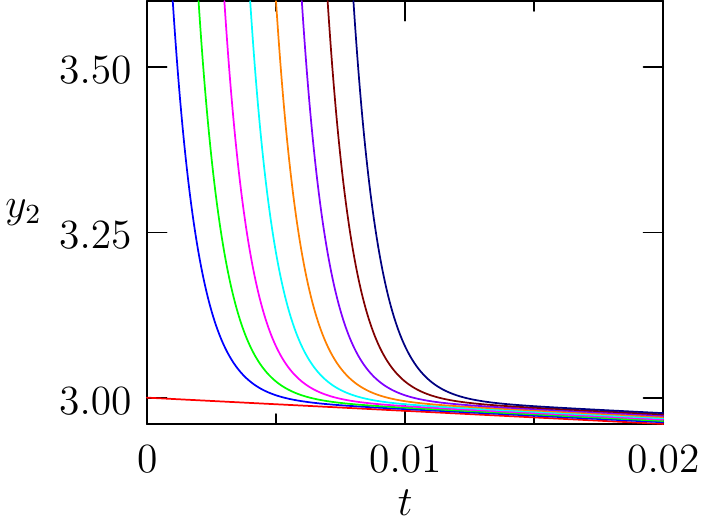}
    \caption{$y_2(t)$ curves for \protect\hlink{System 2} (zoomed).}
    \label{fig-y2-P2-zoom}
\end{subfigure}%
\\[\smallskipamount]
\begin{subfigure}{.5\textwidth}
    \centering
    \includegraphics[width=\linewidth,height=0.2\textheight,keepaspectratio]{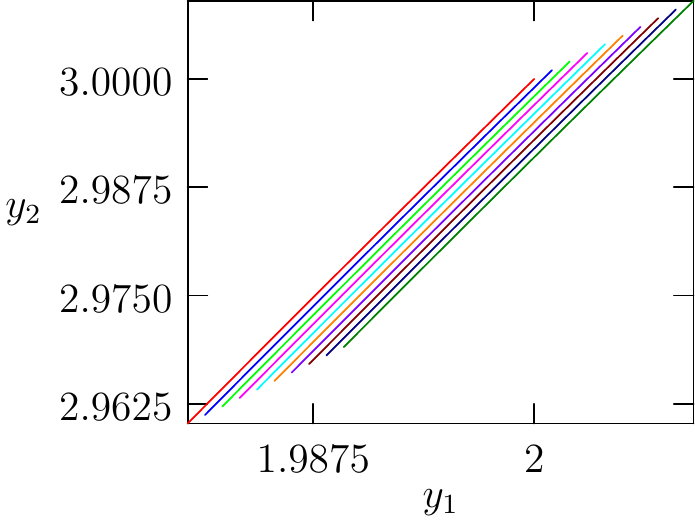}
    \caption{Phase curves for \protect\hlink{System 1} (zoomed).}
    \label{fig-phase-P1-zoom}
\end{subfigure}%
\begin{subfigure}{.5\textwidth}
    \centering
    \includegraphics[width=\linewidth,height=0.2\textheight,keepaspectratio]{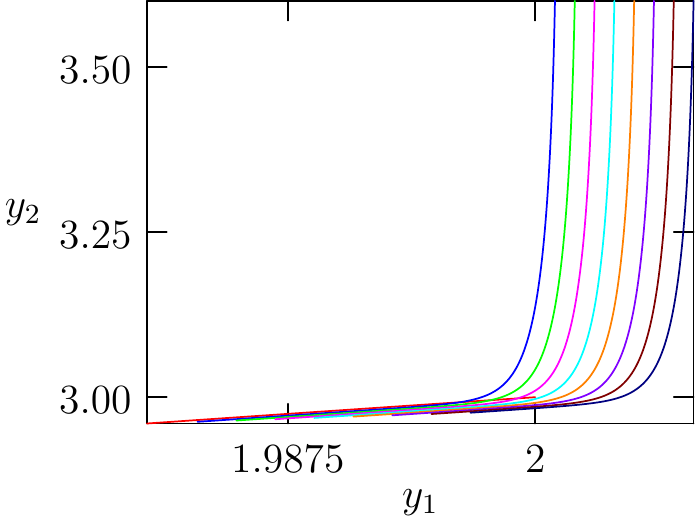}
    \caption{Phase curves for \protect\hlink{System 2} (zoomed).}
    \label{fig-phase-P2-zoom}
\end{subfigure}%
\caption{Plots of nearby solution curves for \protect\hlink{System 1} and \protect\hlink{System 2} zoomed in.}
\label{fig-soln-curves-zoom}
\end{figure}
%Lambert suggests zooming in on some interval, say
%$[0,0.001]$, as in Figure~\ref{fig-y2-P1-zoom} and~\ref{fig-y2-P2-zoom}, so we
%can better observe the rate at which the nearby curves of the
%second component of the solutions for these systems approach each
%other.
%We notice that the nearby curves of the second component of System 2 shown in
%Figure~\ref{fig-y2-P2-zoom} approach
As suggested by Lambert, 
to help us generalize the discussion of stiffness from the case of a
single ODE, 
%and~\ref{fig-phase-P2-zoom},
we zoom in on the
initial evolution in Figure~\ref{fig-soln-curves-zoom}.
While the nearby solutions for $y_2$ are seen
in Figure~\ref{fig-y2-P1-zoom} to be nearly parallel
to each other for (nonstiff)~\hlink{System 1},
they are seen in Figure~\ref{fig-y2-P2-zoom}
to approach each other at a steep angle for (stiff) \hlink{System 2}.
This anomaly can also be observed in the phase curves of
\hlink{System 2} in~\ref{fig-phase-P2-zoom}.

For example, for some $t_n$, the exact value of the solution for
\hlink{System 2} is
$y(t_n)=[1.998, 2.99]$, while the numerical algorithm estimates it
with some error as $y_n=[1.998,3.01]$. The difference
between $y(t_n)$ and $y_n$ is negligible, but the difference
between the slope field vector at $y(t_n)$ and the slope field
vector at $y_n$ is large (see Figure~\ref{fig-field-P2-zoom}). In contrast,
the corresponding vectors in Figure~\ref{fig-field-P1-zoom} are
closely aligned. Since these vectors form
the right-hand side of the ODE at $y(t_n)$ and $y_n$,
we see that~\hlink{System 2}
can pose problems for explicit Runge--Kutta methods.
A common but inefficient remedy is to reduce the time step so that the
vectors $f(t_n,y_n)$ and $f(t_n,y(t_n))$ become more aligned.

\begin{figure}
\centering
\begin{subfigure}{.5\textwidth}
    \centering
    \includegraphics[width=\linewidth]{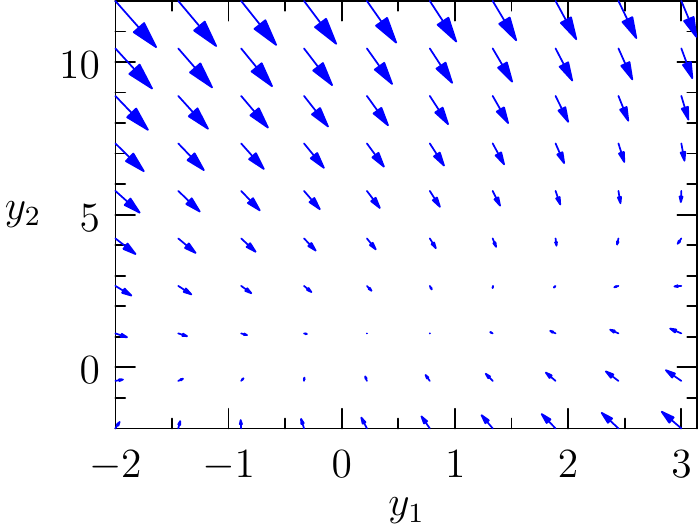}
    \caption{Slope field of \protect\hlink{System 1}.}
    \label{fig-field-P1}
\end{subfigure}%
\begin{subfigure}{.5\textwidth}
    \centering
    \includegraphics[width=\linewidth]{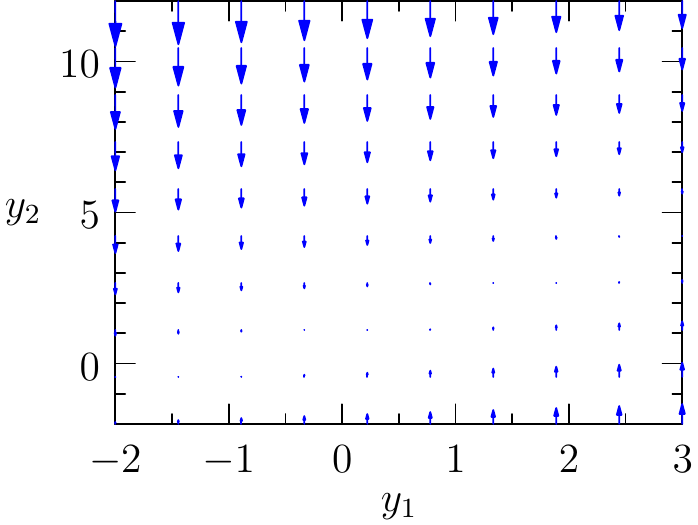}
    \caption{Slope field of \protect\hlink{System 2}.}
    \label{fig-field-P2}
\end{subfigure}%
\\[\smallskipamount]
\begin{subfigure}{.5\textwidth}
    \centering
    \includegraphics[width=\linewidth]{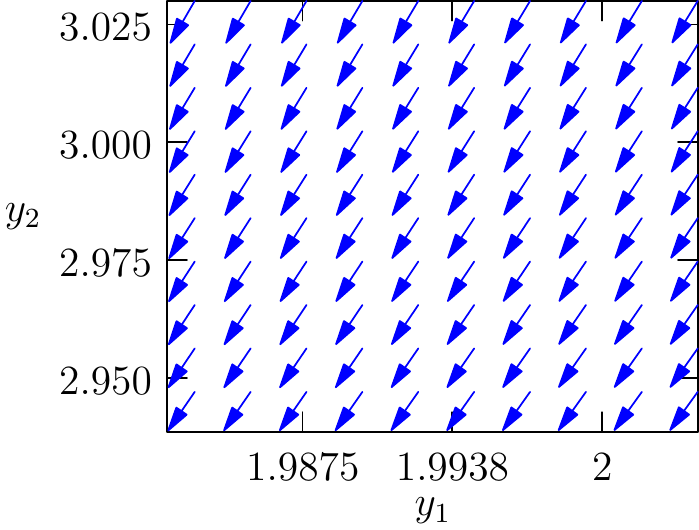}
    \caption{Slope field of \protect\hlink{System 1} (zoomed).}
    \label{fig-field-P1-zoom}
\end{subfigure}%
\begin{subfigure}{.5\textwidth}
    \centering
    \includegraphics[width=\linewidth]{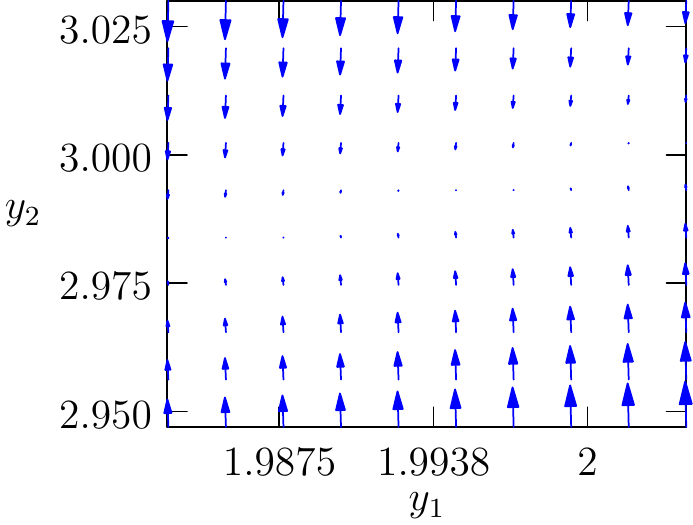}
    \caption{Slope field of \protect\hlink{System 2} (zoomed).}
    \label{fig-field-P2-zoom}
\end{subfigure}%
\caption{Slope field plots for \protect\hlink{System 1} and \protect\hlink{System 2}.}
\label{fig-fields}
\end{figure}

We can now investigate the last statement considered by Lambert
\cite{Lambert91} and attributed to Curtiss and Hirschfelder
\cite{Curtiss52},
which will eventually lead us to our ultimate definition of stiffness.
\begin{statement}
    \label{stiffness}
    A system is said to be stiff in a given interval of time if, in
    that interval, the neighbouring solution curves approach the
    solution curve at a rate which
    is very large in comparison with the rate at which the solution
    varies.
\end{statement}
As Lambert points out for Statement~\ref{stiffness-Lambert},
Statement~\ref{stiffness} also includes the idea that the
stiffness of a system will depend on where in the phase space
the numerical integration takes place.
The reason why Lambert does not adopt
Statement~\ref{stiffness} is that it requires
knowledge of at least two solutions of the
system in order to decide whether nearby curves approach
the desired solution curve at a fast or slow rate. 
Based on the connection we made to the slope field, however, we can
translate the geometric phenomenon of nearby curves approaching at a
fast rate to the more analytic interpretation that the function
$f(t,y)$ has a large Lipschitz constant \cite{Curtiss52}:
\begin{equation*}
    \sup_{\genfrac{}{}{0pt}{2}{t_1\neq t_2}{y_1\neq y_2}} \frac{\abs{f(t_1,y_1) - f(t_2,y_2)}}
                                {\abs{(t_1,y_1) - (t_2,y_2)}}.
\end{equation*}
Essentially, stiff systems are those for which ``a small
change in $y$ leads to a large change in $f(t,y)$''
\cite{Lambert91}. Lambert argues against this statement
as well since it is not apparent
what critical value the Lipschitz constant should
compared to.
The original interest in defining stiffness was solely to
avoid wasting limited resources in solving stiff
systems with explicit methods. However, we are interested in refining
the definition of stiffness further to allow us to specifically detect
linear stiffness (numerical stiffness coming from the linear source term
of the ODE system), so that we can apply exponential
integrators as needed.

Curtiss and Hirschfelder~\cite{Curtiss52} attribute stiffness in a
one-dimensional system
to a drastic change in the slope field across a particular solution
curve. Generalizing this idea to two dimensions,
we consider the equation
\begin{equation}
    \label{eq-diff-y^2}
    \begin{bmatrix}
        y_1'\\
        y_2'
    \end{bmatrix}
    =
    \begin{bmatrix}
        -2 & 1\\
        998 & -999
    \end{bmatrix}
    \begin{bmatrix}
        y_1\\
        y_2
    \end{bmatrix}
    +
    \begin{bmatrix}
        y_1^2\\
        y_2^2
    \end{bmatrix}.
\end{equation}
The solution lives in
a three-dimensional space, with the axes being $t$, $y_1$, and $y_2$.
We first identify the {\it fixed curves\/}
of the components of the system of ODEs~\eqref{eq-diff-y^2}: denote by
$C_1$ the curve corresponding to
$dy_1/dt=0$ and by $C_2$ the curve corresponding to $dy_2/dt=0$.
Pick one of them, say $C_2$. Fix $y_1$ and calculate the
corresponding $y_2$ from the equation for $C_2$. Consider some test
values
$\tilde y_{2a} \in (y_2-\varepsilon ,y_2)$ and
$\tilde y_{2b} \in (y_2, y_2+\varepsilon)$ for some sufficiently
small $\varepsilon$.
The system is stiff if the
slope field vector at $(y_1,\tilde y_{2a})$ or $(y_1,\tilde y_{2b})$
is not aligned with the slope field vector at $(y_1,y_2)$.
We repeat the process for $C_1$.
The equations for $C_1$ and $C_2$ are
\begin{align}
    &C_1:\quad 0 = -2y_1 + y_2 + y_1^2,\\
    &C_2:\quad 0 = 998y_1 - 999y_2 + y_2^2.
\end{align}
We plot these curves and the slope field in Figure~\ref{fig-fields-ysq}.
We note in this case that stiffness manifests itself only in some
parts of the phase space: \emph{stiffness can be a local phenomenon.}
\begin{figure}
\centering
\begin{subfigure}{.5\textwidth}
    \centering
    \includegraphics[width=\linewidth]{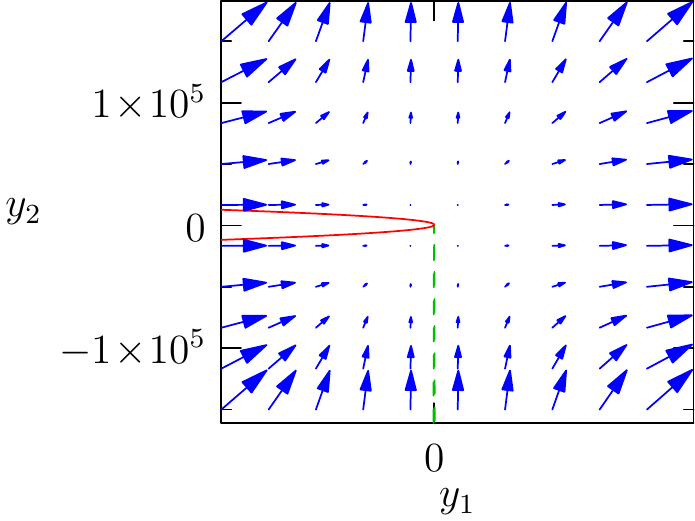}
    \caption{Slope field of~\eqref{eq-diff-y^2}.}
    \label{fig-field-ysq}
\end{subfigure}%
\begin{subfigure}{.5\textwidth}
    \centering
    \includegraphics[width=\linewidth]{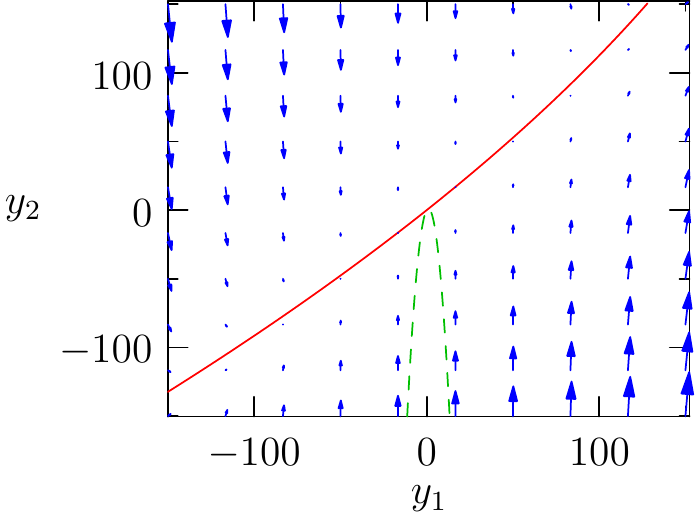}
    \caption{Slope field of~\eqref{eq-diff-y^2} zoomed.}
    \label{fig-field-ysq-zoom}
\end{subfigure}%
\caption{Slope field plots of~\eqref{eq-diff-y^2} showing the fixed
  curves $C_1$ (dashed green) and $C_2$ (solid red).}
\label{fig-fields-ysq}
\end{figure}

In contrast, the system obtained by removing the nonlinear terms from~\eqref{eq-diff-y^2},
\begin{equation}
    \label{eq-diff-exp}
    \begin{bmatrix}
        y_1'\\
        y_2'
    \end{bmatrix}
    =
    \begin{bmatrix}
        -2 & 1\\
        998 & -999
    \end{bmatrix}
    \begin{bmatrix}
        y_1\\
        y_2
    \end{bmatrix},
\end{equation}
becomes stiff for every initial condition.
Here, the equations for the curves $C_1$ and $C_2$ are
\begin{align}
    &C_1:\quad 0 = -2y_1 + y_2,\\
    &C_2:\quad 0 = 998y_1 - 999y_2;
\end{align}
these are plotted together with the corresponding slope field in Figure~\ref{fig-fields-exp}.
\begin{figure}
\centering
\begin{subfigure}{.5\textwidth}
    \centering
    \includegraphics[width=\linewidth]{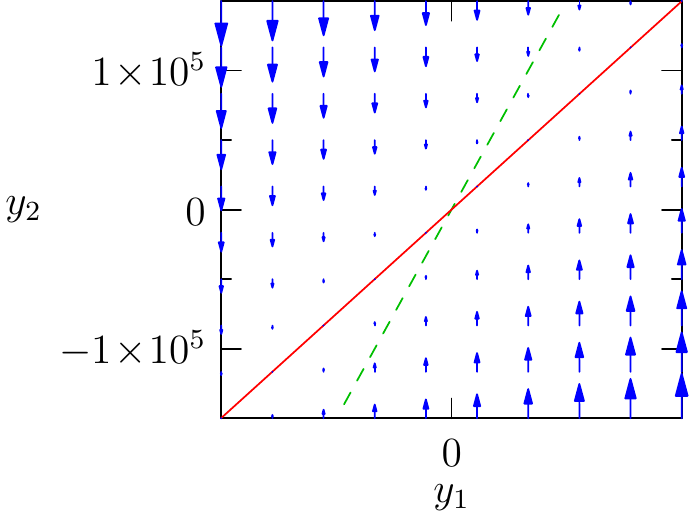}
    \caption{Slope field of~\eqref{eq-diff-exp}.}
    \label{fig-field-exp}
\end{subfigure}%
\begin{subfigure}{.5\textwidth}
    \centering
    \includegraphics[width=\linewidth]{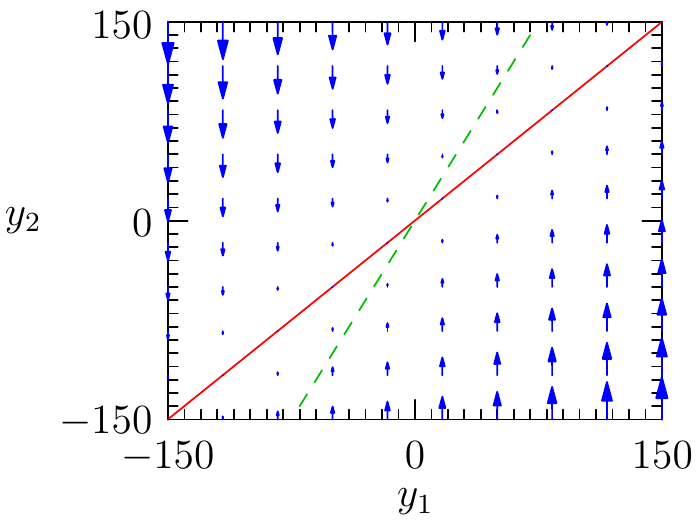}
    \caption{Slope field of~\eqref{eq-diff-exp} zoomed.}
    \label{fig-field-exp-zoom}
\end{subfigure}%
\caption{Slope field plots of~\eqref{eq-diff-exp} showing the fixed
  curves $C_1$ (dashed green) and $C_2$ (solid red).}
\label{fig-fields-exp}
\end{figure}
This is another confirmation that stiffness cannot be measured by
properties of the linear operator only and makes it
crucial to have a definition that can recognize if the equation
is stiff or not in the phase space region of interest.

The idea of testing the behaviour of the source function near fixed
curves provides a fast but heuristic way of determining if an equation
is stiff in a particular region.
However, Cartwrite has proposed an alternative, more theoretical,
definition of stiffness based on local Lyapunov exponents and
curvature \cite{Cartwright99}.
Although he is mostly concerned with defining stiffness
in chaotic systems, part of his work aligns well with the
discussion here.
In fact, we consider Cartwright's
definition to be the preferred definition of stiffness as it is
local and provides a way of quantifying the rate at which nearby
solution curves approach the exact solution.

Stiff systems can be recognized by how fast 
nearby solutions approach a fixed curve and, as mentioned
above, what characterizes the rate of separation or contraction
of solution curves of a system of ODEs are the
Lyapunov exponents.
Since Lyapunov exponents give a global picture of the
phase space, to investigate what happens locally
we will need \emph{local Lyapunov exponents}. Let
$\sigma_i(t)$ be the principal axes of an ellipsoidal ball
evolving in time in phase space. The $i^{th}$ local Lyapunov
exponent is
\begin{equation}
    \gamma_i(\tau,t)=\lim_{\sigma_i(\tau)\rightarrow 0}
                    \frac{1}{\tau}\log \frac{\sigma_i(t+\tau)}{\sigma_i(t)}.
\end{equation}
We hinted earlier that while nearby solutions play a major role in detecting
stiffness in a system of ODEs, the wiggliness of the
solution curve at the section of interest in the phase space also
plays an important role. Cartwright quantifies the latter by using
the curvature $\kappa = y''{(1+y'^2)^{-3/2}}$ of the solution $y$.
Cartwright states
that ``a system is stiff in a given interval if in that interval the
most negative local Lyapunov exponent is large, while the
curvature of the solution is small'' \cite{Cartwright99}.
He quantifies stiffness by the ratio
\begin{equation}
  R_{nl}=\frac{\displaystyle \left|\min_{1\leq i\leq n}\gamma_i(\tau,t)\right|}{\kappa(t)},
\end{equation}
but he does not compare the ratio to anything and hence leaves it
open to Lambert's argument against Statement~\ref{stiffness}.

Cartwright also points out that the ratio $R_{nl}$ could
be averaged over the
trajectory to yield a global measure of stiffness, but we have already
seen that stiffness can be a local phenomenon.
Therefore, we slightly modify Cartwright's definition to
\begin{definition}
    \label{def-stiffness}
    A system is \emph{stiff} in a given interval if in that interval the
    most negative local Lyapunov exponent is much larger in absolute
    value than the curvature of the solution curve.
\end{definition}

In the remainder of this work, we focus on
systems of ODEs of the form \eqref{eq-LF},
such that the stiffness in the sense of
Definition~\ref{def-stiffness} stems mostly from the
term~$Ly$.

\section{Exponential Integrators}\label{exp-integrators}
We now introduce a class of methods known as \emph{exponential integrators},
first encountered in \cite{Certaine60}, that avoid linear
stiffness by treating the linear term exactly. 

Returning to~\eqref{eq-LF}, we define the function $G(t)=F(t,y(t))$,
and introduce the integrating factor $t\mapsto e^{tL}$.
As shown in \cite{Zoto23embed}, a change of both independent and
dependent variables allows us to write the exact solution of~\eqref{eq-LF} as
\begin{equation}
    \label{eq-FL-sol-ph}
    y(t_n+h) = e^{-hL}y(t_n)
            + \sum_{k=0}^\infty 
            h^{k+1} \ph_{k+1}(-hL) G^{(k)}(t_n),
\end{equation}
where
\begin{align}
    \ph_k(0) &= \frac{1}{k!},\\
    \ph_0(x) &= e^x,\\
    \ph_{k+1}(x) &= \frac{\ph_k(x)-\frac{1}{k!}}{x} \text{ for }
                    k \geq 0.
\end{align}
Note that~\eqref{eq-FL-sol-ph} agrees  with equation (4.6) of
Ref.~\cite{Hochbruck05} obtained by applying the variation-of-constants
method to \eqref{eq-LF} and Taylor expanding $G$.

Exponential Runge--Kutta methods approximate the infinite sum and
derivatives of $G$ in
\eqref{eq-FL-sol-ph}. The simplest approximation,
called the exponential Euler method, truncates the sum
after the first term:
\begin{equation}
    y(t_n+h) = \ph_0(-hL)y_n + h\ph_1(-hL)F(t_n,y_n).
\end{equation}
The exponential Euler method solves~\eqref{eq-LF}
exactly whenever $F(t,y)$ is constant but reduces to the explicit
Euler method in the \emph{classical limit} $L\to 0$.

As shown by Hochbruck and Ostermann \cite{Hochbruck05},
great care must be taken when deriving a higher-order exponential
integrator
to ensure that it retains its design order when applied to
stiff problems. Hochbruck and Ostermann demonstrate that
several fourth-order exponential integrators in the literature
exhibit an order reduction when applied to a particular test problem.
For example, the stiff order of the scheme ETD4RK of Cox and Matthews
\cite{Cox02} can drop from four to two.
For consistency of nomenclature, we refer to this
method~\hlink{ERK4CM} and give its Butcher tableau in Table~\ref{t-ERK4CM}.
\begin{table}[htbp]
  \hypertarget{ERK4CM}\
\[
\renewcommand\arraystretch{1.2}
\begin{array}
{c|cccc}
0\\
\frac{1}{2}
    & \frac{1}{2}\ph_1\(-\frac{hL}{2}\)\\
\frac{3}{4}
    & 0
    & \frac{1}{2}\ph_1\(-\frac{hL}{2}\)\\
1
    & \frac{1}{2}\ph_1\(-\frac{hL}{2}\)\Bigl(\ph_0\(-\frac{hL}{2}\)
                                                        -1\Bigr)
    & 0
    & \ph_1\(-\frac{hL}{2}\)\\
\hline
1
    & \ph_1-3\ph_2+4\ph_3
    & 2\ph_2-4\ph_3
    & 2\ph_2-4\ph_3
    & 4\ph_3-\ph_2
\end{array}
\]
    \caption{ERK4CM tableau, where $\ph_i=\ph_i(-hL)$.}
    \label{t-ERK4CM}

\end{table}

Similarly, the exponential integrator of Krogstad \cite{Krogstad05} in
Table~\ref{t-ERK4K}, which we denote~\hlink{ERK4K}, can suffer an
order reduction from four to three.

\begin{table}[htbp]
  \hypertarget{ERK4K}\
\[
\renewcommand\arraystretch{1.2}
\begin{array}
{c|cccc}
0\\
\frac{1}{2} 
    & \frac{1}{2}\ph_1\(-\frac{hL}{2}\)\\
\frac{1}{2} 
    & \frac{1}{2}\ph_1\(-\frac{hL}{2}\)-\ph_2\(-\frac{hL}{2}\)
    & \ph_2\(-\frac{hL}{2}\)\\
1
    & \ph_1-2\ph_2
    & 0
    & 2\ph_2\\
\hline
1
    & \ph_1-3\ph_2+4\ph_3
    & 2\ph_2-4\ph_3
    & 2\ph_2-4\ph_3
    & 4\ph_3-\ph_2
\end{array}
\]
    \caption{ERK4K tableau, where $\ph_i = \ph_i(-hL)$.}
    \label{t-ERK4K}
\end{table}

Hochbruck and Ostermann \cite{Hochbruck05}
derived a set of stiff-order conditions that are sufficient to
prevent such order reductions. They propose the five-stage method
shown in Table~\ref{t-ERK4HO5}, which we denote~\hlink{ERK4HO5}.

\begin{table}[htbp]
  \hypertarget{ERK4HO5}\
\[
\renewcommand\arraystretch{1.2}
\begin{array}
{c|ccccc}
0\\
\frac{1}{2} 
    & \frac{1}{2}\ph_1\(-\frac{hL}{2}\)\\
\frac{1}{2} 
    & \frac{1}{2}\ph_1\(-\frac{hL}{2}\) - \ph_2\(-\frac{hL}{2}\)
    & \ph_2\(-\frac{hL}{2}\)\\
1
    & \ph_1-2\ph_2
    & \ph_2
    & \ph_2\\
\frac{1}{2}
    & \frac{1}{2}\ph_1\(-\frac{hL}{2}\)-2a_{31}-a_{33}
    & a_{31}
    & a_{31}
    & \frac{1}{4}\ph_2\(-\frac{hL}{2}\)-a_{31} \\
\hline
1
    & \ph_1-3\ph_2+4\ph3
    & 0
    & 0
    & -\ph_2+4\ph_3
    & 4\ph_2-8\ph_3\\
\end{array}
\]
\begin{align*}
     \ph_i &= \ph_i(-hL),\\
     a_{31} &= \frac{1}{2}\ph_2\(-\frac{hL}{2}\)-\ph_3
     +\frac{1}{4}\ph_2 - \frac{1}{2}\ph_3\(-\frac{hL}{2}\).
\end{align*}
    \caption{ERK4HO5 tableau.}
    \label{t-ERK4HO5}
\end{table}

State-of-the-art numerical methods use adaptive time-stepping
to efficiently allocate computational resources. These methods adjust
the time step to keep the estimated local error within prescribed
bounds. The error estimate is computed as the difference between high-
and low-order approximations. Embedded methods, which share some of the
sample function evaluations between the two approximations, can yield
very efficient adaptive time stepping schemes. As pointed out in
Ref.~\cite{Zoto23embed}, it is important that embedded method be
\emph{robust:} the order of the low-order approximation should never
equal the order $n$ of the high-order approximation for any function~$G(t)$
with a nonzero derivative of order less than $n$.
A robust fourth-order embedded exponential integrator
called~\hlink{ERK43ZB} is presented in Table~\ref{t-ERK43ZB}~\cite{Zoto23embed}.

\begin{table}
\[
\renewcommand\arraystretch{1.2}
\begin{array}
{c|cccccc}
0\\
\frac{1}{6}
    & \frac{1}{6}\ph_1\(-\frac{hL}{6}\)\\
\frac{1}{2}
    & \frac{1}{2}\ph_1\(-\frac{hL}{2}\) - a_{11}
    & a_{11}\\
\frac{1}{2}
    & \frac{1}{2}\ph_1\(-\frac{hL}{2}\)-a_{21}-a_{22}
    & a_{21}
    & a_{22}\\
\hline
1
    & \ph_1-a_{31}-a_{32}-a_{33}
    & a_{31}
    & a_{32}
    & a_{33} \\
1
    & \ph_1-\frac{67}{9}\ph_2+\frac{52}{3}\ph_3
    & 8\ph_2-24\ph_3
    & \frac{26}{3}\ph_3-\frac{11}{9}\ph_2
    & a_{43}
    & a_{44}
\end{array}
\]
\begin{align*}
    \ph_i  &= \ph_i(-hL)\\
    a_{11} &= \frac{3}{2}\ph_2\(-\frac{hL}{2}\)
        + \frac{1}{2}\ph_2\(-\frac{hL}{6}\)\\
    a_{21} &= \frac{19}{60}\ph_1 + \frac{1}{2}\ph_1\(-\frac{hL}{2}\)
    +\frac{1}{2}\ph_1\(-\frac{hL}{6}\)\\
           &+2\ph_2\(-\frac{hL}{2}\)
           +\frac{13}{6}\ph_2\(-\frac{hL}{6}\)
    +\frac{3}{5}\ph_3\(-\frac{hL}{2}\)\\
    a_{22} &= -\frac{19}{180}\ph_1
    - \frac{1}{6}\ph_1\(-\frac{hL}{2}\)
    -\frac{1}{6}\ph_1\(-\frac{hL}{6}\)\\
            &-\frac{1}{6}\ph_2\(-\frac{hL}{2}\)
    +\frac{1}{9}\ph_2\(-\frac{hL}{6}\)
    -\frac{1}{5}\ph_3\(-\frac{hL}{2}\)\\
    a_{33} &= \ph_2 + \ph_2\(-\frac{hL}{2}\) - 6\ph_3
            - 3\ph_3\(-\frac{hL}{2}\)\\
    a_{31} &= 3\ph_2 - \frac{9}{2}\ph_2\(-\frac{hL}{2}\)
    - \frac{5}{2}\ph_2\(-\frac{hL}{6}\) + 6a_{33} + a_{21}\\
    a_{32} &= 6\ph_3 + 3\ph_3\(-\frac{hL}{2}\)
    -2a_{33}+a_{22}\\
    a_{43} &= \frac{7}{9}\ph_2 - \frac{10}{3}\ph_3,\qquad
    a_{44} = \frac{4}{3}\ph_3 - \frac{1}{9}\ph_2
\end{align*}
    \caption{ERK43ZB tableau.}
    \label{t-ERK43ZB}
  \hypertarget{ERK43ZB}\
\end{table}

\section{Schur decomposition}\label{schur}
Exponential integrators are invariant under the transformation
of~\eqref{eq-LF} to the 
autonomous form of the equation (where~$f$ does not have an
explicit time dependence), by introducing a new independent variable:
\begin{equation}
    \label{eq-LFS}
    \frac{dy}{dt} = F(y)-Ly.
\end{equation}
Applying an exponential Runge--Kutta method to such a system
requires the evaluation of $\ph_k(c_jhL)$ for certain values of $k$
and $j$. Since $L$ is assumed to be a general matrix, these
functions are related to the matrix exponential.
Methods for efficiently calculating
such matrix functions, and the matrix exponential in particular,
are an active field of research.
In our testing, we chose to implement a scaling and squaring algorithm,
followed by a Pad\'e approximant,
for calculating the matrix exponential \cite{Moler03}\cite{Higham08}.
Although slow, this is a reliable technique for calculating
the $\ph_k$ functions for a general matrix $L$. In particular cases,
special properties of $L$ may be used to devise more computationally
efficient methods. For example, if $L$ is a
sparse matrix, it is worth implementing a Krylov subspace method.
A short description of Krylov subspace methods (as well as other
methods such as Chebyshev methods, Leja interpolation, and contour
integrals) in relation to exponential RK methods is given in
\cite{Hochbruck10}. Instead of examining existing methods for computing matrix exponentials in greater detail, we propose a transformation to the equation such that the matrix in the term that is
treated exactly by the ERK method is diagonal.

An important practical application of exponential integrators are
PDEs containing a linear term $Ly$, where $L$ is a Laplacian. If
spectral transforms are used to 
convert spatial derivatives to algebraic expressions, the Laplacian
becomes a diagonal matrix.
Calculating the exponential of a diagonal matrix is straightforward to implement and computationally inexpensive.
To avoid loss of accuracy due to finite numerical precision,
truncations of Taylor series should be used when
evaluating $\ph_k(x)$ near $0$ for $k > 0$ \cite{Bowman05}.

Other applications use finite differences to approximate the Laplacian
as a nondiagonal discretized spatial operator. In these cases, the
difficulty of accurately computing the various matrix $\ph_k$ functions has
discouraged many researchers from using adaptive exponential integrators.
Recognizing the computational advantages of the diagonal case,
it would seem reasonable when $L$ is diagonalizable to
compute a one-time change of basis that diagonalizes $L$;
that basis can then be reused for computing matrix functions of $c_jhL$
for arbitrary values of $h$.
However, diagonalization is well known to become numerically unstable
when eigenvalues coalesce.
Moreover, not all matrices are diagonalizable.
Instead of trying to diagonalize $L$, one can find
its Schur decomposition
\begin{equation}
    L=UTU^\dag,
\end{equation}
\let\dag\dagger
where $U^\dag$ denotes the conjugate transpose of the \emph{unitary} matrix
$U$ (so that $U^\dag=U^{-1}$) and $T$ is an upper triangular matrix.
Furthermore, we can write $T=D+S$, where~$D$ is a diagonal matrix and
$S$ is a strictly upper triangular matrix. Equation~\eqref{eq-LFS}
becomes
\begin{equation}
    \frac{dy}{dt} + U(D+S)U^\dag y = F(t,y).
\end{equation}
On multiplying by $U^\dag$ on the left we obtain
\begin{equation}
    \frac{d(U^\dag y)}{dt} + (D+S)U^\dag y = U^\dag F(t,y),\\
\end{equation}
or, in terms of the transformed variable $Y=U^\dag y$, 
\begin{equation}
    \frac{dY}{dt} + DY = U^\dag F(t,UY) - SY,\label{eq-DSF}
\end{equation}
By applying this transformation, we avoid
working with exponentials of a full matrix in favour of
exponentials of a diagonal matrix.
The main trade-off is that we have to compute the Schur
decomposition for the matrix $L$, but that is only done once
and the longer the interval for the time integration, the
more worthwhile this investment becomes.
The second drawback is that some part that could have been treated
exactly is now treated numerically and this could contribute
to the overall error. In addition, we have to do two matrix
multiplications at each step because the nonlinearity $F(t,y)$ is
evaluated in the initial space. Although the efficiency gained in
calculating the $\ph_k$ functions is more than enough to
compensate for the drawback of two added multiplications per step
and the potential of added numerical error, there is another
advantage to implementing exponential RK methods in this way. The
$\ph_k$ functions are now diagonal matrices and can thus be stored
as vectors. This is a large improvement in memory usage as even
for sparse matrices $L$, the matrix $\ph_k$ functions are
general full matrices requiring extra storage. Furthermore,
the ERK methods (and classical RK methods) work by multiplying the
weights of the method by previously computed approximations of the
vector $y$. In the case of ERK methods, the weights are linear
combinations of matrix functions and hence matrices themselves. By
implementing the Schur decomposition and being able to work with
weights that are diagonal matrices, we have replaced all the needed
matrix-vector multiplications with computationally cheap vector
dot products.

With the optimization afforded by Schur decomposition, the use of
embedded ERK methods for step size adjustment becomes computationally viable,
even when $L$ is a nondiagonal matrix.
An adaptive exponential method requires recalculating the weights
(and corresponding $\ph_k$ functions) every time that the step size
is adjusted. However, since these are now functions of diagonal
matrices, there is no longer a huge computational cost to bear.
As in the
case of fixed step size, the Schur decomposition of $L$ only
needs to be performed once, so depending on the duration of
the integration, the cost of the decomposition will
typically be negligible.

Since many matrices encountered in practice are normal,
the following result shows in these cases that the Schur decomposition
technique not only removes linear stiffness from the problem, but
will still handle the linear term exactly (since $S=0$).
\begin{theorem}
    The triangle matrix in the Schur decomposition of a normal matrix is diagonal.
\end{theorem}
\begin{proof}
    Assume $L$ is a \emph{normal} matrix:
    \begin{equation}
        LL^\dag = L^\dag L.
    \end{equation}
    The \emph{Schur decomposition} of $L$ and $L^\dag$ are
    \begin{equation}
        L=U^\dag T U \quad \text{ and }
        \quad L^\dag = U^\dag T^\dag U,
    \end{equation}
    where $T$ is a triangular matrix and $U$ is a \emph{unitary} matrix,
    so that $U^\dag = U^{-1}$.
    Hence
    \begin{equation}
        U^\dag T U U^\dag T^\dag U = U^\dag T^\dag U U^\dag T U,
    \end{equation}
    which reduces to
    \begin{equation}
        T T^\dag = T^\dag T.
    \end{equation}
    This means that the triangular matrix $T$ resulting from the
    Schur decomposition of~$L$ is normal. An inductive argument
    shows that it must then be diagonal \cite{Prasolov94}.
\end{proof}

In the general case, the strictly upper triangular matrix $S$ resulting
from the Schur decomposition will be non-zero. We now show
that the term $Sy$ does not incorporate any of the
stiffness inherent in the linear term $Ly$.
On defining the integrating factor $I(t)=e^{tD}$ and
\def\yt{\Tilde y}
\def\St{\Tilde S}
$\Tilde y(t)=I(t)y(t)$, we can transform \eqref{eq-DSF} in the
autonomous case to
\begin{equation}
    \label{eq-S}
    \frac{d\yt}{dt}=I(t)U^\dag F(UI^{-1}(t)\yt)-\St\yt,
\end{equation}
where $\St=I(t)S I^{-1}(t)$ is an $m{\times}m$ strictly upper
triangular matrix.
For systems of the form \eqref{eq-LFS} where the stiffness only enters through the linear term
$Ly$ and not through $F(y)$, the first term on the right-hand side of
\eqref{eq-S} will not contribute any additional stiffness. To analyze
\eqref{eq-S} we first consider the case $F=0$, when it reduces to the
triangular system of equations
\begin{equation}
    \frac{d\yt_i}{dt}=\sum_{j=i+1}^{m}\St_{ij}\yt_j \text{ for }
    i=1,\dots,m-1
    \quad \text{ and } \quad \frac{d\yt_m}{dt}=0,
\end{equation}
which can be solved recursively to obtain the general solution as
a polynomial in~$t$. Recalling that stiffness
arises only when nearby solution curves approach the
solution curve of interest at exponentially fast rates, we deduce
that since polynomials cannot approach each other exponentially
fast, the system of equations is not stiff.
Such ODE systems can even be
solved exactly by a classical Runge--Kutta method whose degree
is greater than or equal to the degree of each of the solution polynomials.
By linear superposition, it follows that~\eqref{eq-S} is not stiff even
when~$F$ is linear and, in particular, when $F$ is constant.
That is, all of the linear stiffness in~\eqref{eq-DSF} is contained within
the diagonal term $DY$.

% Matrix exponential computational details

\section{Examples and applications}\label{applications}
Let us have a look at an example where the matrix $L$ is upper
triangular. First, we will solve the system by treating the full
linear term by an exponential RK method. Then, we will split $L$
into the sum of a diagonal matrix $D$
and a strictly upper triangular matrix $S$, allowing us to treat the
diagonal term exactly and the rest numerically. Consider the system
\begin{equation}
    \frac{dy}{dt} = -Ly= -Dy - Sy,
\end{equation}
where
\begin{equation}
D=-
\begin{bmatrix}
    a &0 &0\\
    0 &d &0\\
    0 &0 &f
\end{bmatrix},
S=-
\begin{bmatrix}
    0 &b &c\\
    0 &0 &e\\
    0 &0 &0
\end{bmatrix}
\end{equation}
and the eigenvalues $a$, $d$, and $f$ of $L$ are distinct.
The general solution of this system of ODEs is
\begin{equation}
\begin{bmatrix}
    y_1\\
    y_2\\
    y_3
\end{bmatrix}
=
\begin{bmatrix}
    k_1e^{-at} + \frac{bek_3e^{-ft}}{(f-a)(f-d)}
    + \frac{bk_2e^{-dt}}{d-a} + \frac{ck_3e^{-ft}}{f-a}\\
    k_2e^{-dt}+\frac{ek_3e^{-ft}}{f-d}\\
    k_3e^{-ft}
\end{bmatrix},
\end{equation}
where the constants $k_1$, $k_2$, and $k_3$ are fixed by the chosen
initial condition. This algebraic example allows us to run
many tests with ease, say

\begin{equation}
\label{eq-exampleS2}
L=-
\begin{bmatrix}
    1 &2 &7\\
    0 &75 &8\\
    0 &0 &15
\end{bmatrix},\quad
y(0)=
\begin{bmatrix}
    1\\
    1\\
    1
\end{bmatrix}.
\end{equation}
For this problem, we compare in Figure~\ref{fig-errors-S-smaller}
the error in~\Hlink{ERK4HO5}{ERK4HO5M}, which is a matrix
implementation of~\hlink{ERK4HO5}
\cite{Hochbruck05} and~\Hlink{ERK4HO5}{ERK4HO5V}, which is an implementation
of the same method with only the diagonal part treated exactly.
Here V stands for ``vector'' since the matrix in the linear term is
taken to be the diagonal matrix $D$ and M stands for ``matrix'' since
the matrix in the linear term is the complete matrix $L$.
The full matrix implementation~\Hlink{ERK4HO5}{ERK4HO5M} is supposed to solve the problem
exactly (since $F=0$)
but is still susceptible to floating point precision error.
Both exponential methods behave as expected
for large time steps. In comparison, we show how the classical
RK4 method fails at large~$h$.
This small system of ODEs
demonstrates the previous argument that stiffness is isolated to the
diagonal term $Dy$.

\begin{figure}[h]
    \centering
    \includegraphics[width=0.49\linewidth]{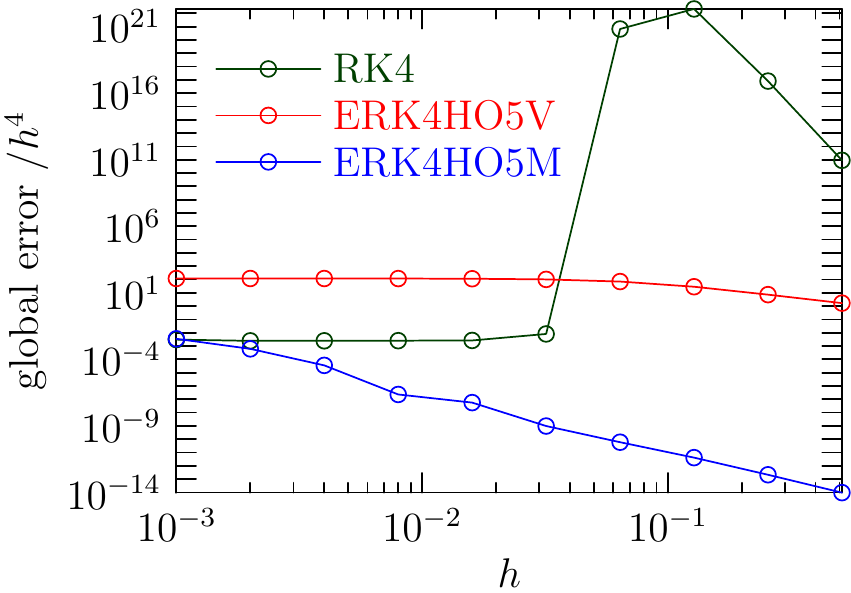}
    \caption{Error at $t=1$ when solving~\eqref{eq-exampleS2} with \protect\Hlink{ERK4HO5}{ERK4HO5M}, \protect\Hlink{ERK4HO5}{ERK4HO5V}, and RK4.}
\label{fig-errors-S-smaller}
\end{figure}

Consider Example 6.2 of \cite{Hochbruck05}:
\begin{equation}
    \label{eq-HO5}
    \frac{\partial y}{\partial t}(x,t)
    -\frac{\partial^2 y}{\partial x^2}(x,t)
    = \int_0^1 y(\bar x,t)\,d\bar x + \Phi(x,t),
\end{equation}
for $x\in [0,1]$ and $t\in [0,1]$, subject to homogeneous
Dirichlet boundary conditions, where the function $\Phi$ is chosen by
substituting the specified exact solution
\begin{equation}\label{eq-exact}
    y(x,t) = x(1-x)e^t
\end{equation}
into~\eqref{eq-HO5}.
This problem can be transformed to a system of ODEs by performing
a centered spatial discretization of the Laplacian and integral.
We approximate the integral with the Simpson method, which in this case
%particular integral in~\eqref{eq-HO5}
can be written as a matrix-vector multiplication.
This means it is a linear term and hence could be fused with the
linear term coming from
the discretized Laplacian. Therefore, all exponential integrators
could solve this problem exactly. Since treating numerically a
part of the equation that can be treated exactly is not a fair
comparison, we modify~\eqref{eq-HO5} to
\begin{equation}
    \label{eq-HO5-mod}
    \frac{\partial y}{\partial t}(x,t)
    -\frac{\partial^2 y}{\partial x^2}(x,t)
    = \int_0^1 y^4(\bar x,t)\,d\bar x + \Phi(x,t),
\end{equation}
where again the function $\Phi$ is calculated by
substituting~\eqref{eq-exact} in~\eqref{eq-HO5-mod}.
We discretized Problem \eqref{eq-HO5-mod} with $200$ spatial grid points.
As in \cite{Hochbruck05}, we calculate the matrix
$\ph_k$ functions with the help of Pad\'e approximants, along with scaling and
squaring.

In Figure~\ref{fig-error-ERK4HO5}, we plot the $L^2$ norm of the global
error at $t=1$ for
the full discretized Laplacian matrix formulation of \hlink{ERK4HO5}
and the optimized implementation where the Laplacian is first reduced to a
diagonal matrix via Schur decomposition.
Figure~\ref{fig-error-ERK43ZB} shows the same situation for the
fourth-order estimate of~\hlink{ERK43ZB}.
Figure~\ref{fig-error-ERK4K} and Figure~\ref{fig-error-ERK4CM}
emphasize that even when the~\hlink{ERK4K} and~\hlink{ERK4CM}
methods are applied to systems with a diagonal linear term, they can
still suffer from order reduction.
\begin{figure}[h]
\centering
\begin{subfigure}{0.49\linewidth}
    \centering
    \includegraphics[width=\linewidth]{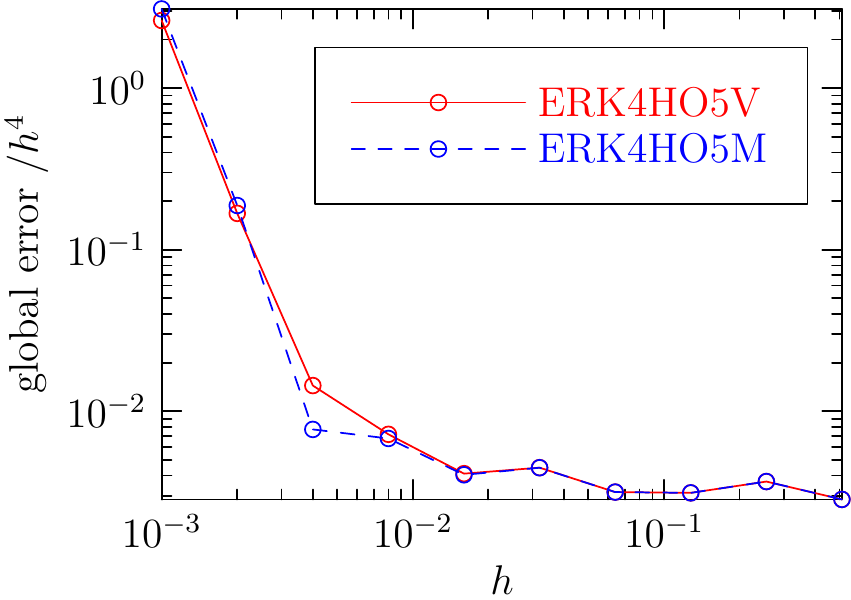}
    \caption{Error comparison for \protect\hlink{ERK4HO5}.}
    \label{fig-error-ERK4HO5}
\end{subfigure}%
\,
\begin{subfigure}{0.49\linewidth}
    \centering
    \includegraphics[width=\linewidth]{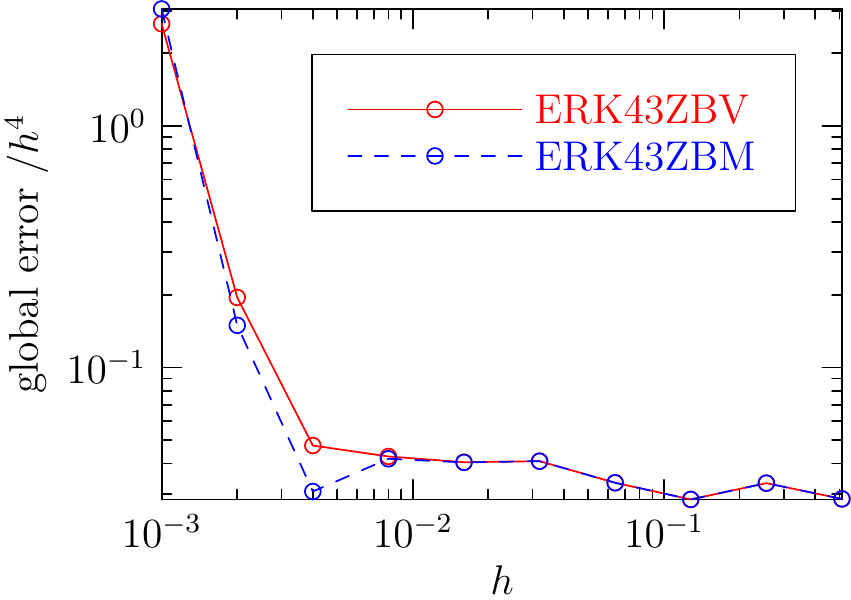}
    \caption{Error comparison for \protect\hlink{ERK43ZB}.}
    \label{fig-error-ERK43ZB}
\end{subfigure}%
\caption{Error comparison when solving~\eqref{eq-HO5-mod} with and without the Schur decomposition.}
\label{fig-error-VM}
\end{figure}

\begin{figure}[h]
\centering
\begin{subfigure}{0.49\linewidth}
    \centering
    \includegraphics[width=\linewidth]{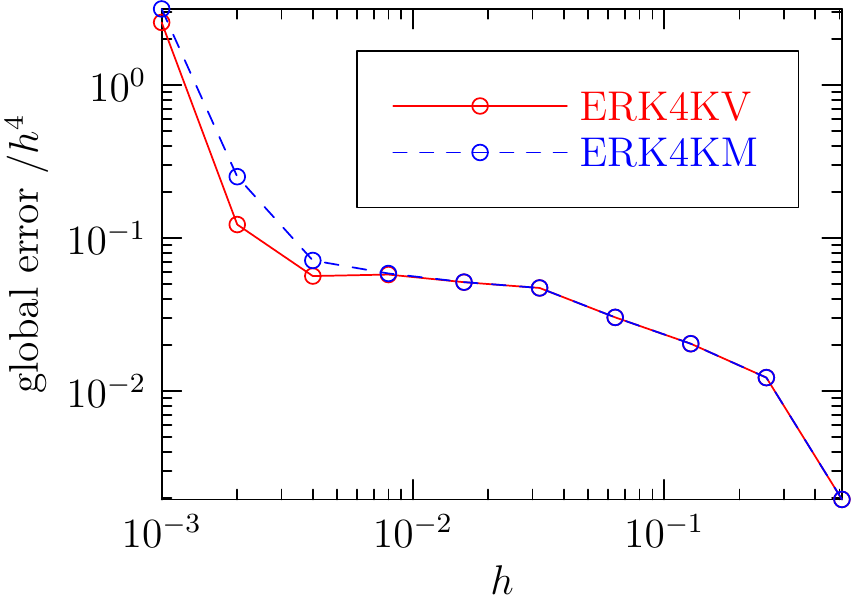}
    \caption{Error comparison for \protect\hlink{ERK4K}.}
    \label{fig-error-ERK4K}
\end{subfigure}%
\,
\begin{subfigure}{0.49\linewidth}
    \centering
    \includegraphics[width=\linewidth]{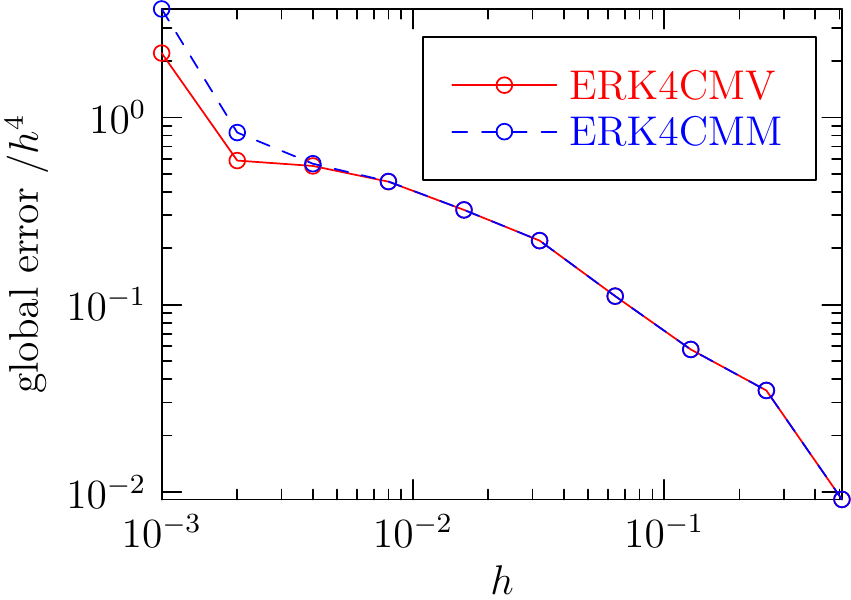}
    \caption{Error comparison for \protect\hlink{ERK4CM}.}
    \label{fig-error-ERK4CM}
\end{subfigure}%
\caption{Error comparison when solving~\eqref{eq-HO5-mod} with and without the Schur decomposition.}
\label{fig-error1-VM}
\end{figure}

To illustrate the impressive performance gain afforded by Schur decomposition,
we considered the equation
\begin{equation}
    \label{eq-HO-6-1}
    \frac{\partial y}{\partial t}(x,t)
    -\frac{\partial^2 y}{\partial x^2}(x,t)
    = \frac{1}{1+y(x,t)^2} + \Phi(x,t),
\end{equation}
with $\Phi$ chosen so that the exact solution
$y(x,t)=10(1-x)x(1+\sin t)+2$ is oscillatory rather than exponential.
We integrated~\eqref{eq-HO-6-1} from $t=0$ to $t=200$ using vector and
matrix formulations of~\hlink{ERK43ZB}, with $3000$ spatial points
and a fixed time step of~0.3. The vector formulation
\Hlink{ERK43ZB}{ERK43ZBV} was found to run about 117 times faster than
the matrix formulation \Hlink{ERK43ZB}{ERK43ZBM}, even after taking
into account the cost of the Schur decomposition, which required about
50\% of the total run time.
Since \hlink{ERK43ZB} is a robust embedded method, it is even more
meaningful to illustrate the practicality of Schur decomposition
using adaptive time stepping on a large problem. For $10\ 000$ spatial
points, Schur decomposition took only about 7\% of the total time required
to integrate~\eqref{eq-HO-6-1} from $t=0$ to $t=20\ 000$.

\section{Conclusion}
Previous attempts at defining stiffness in the literature are inadequate.
The quantitative definition of stiffness given in this work, which
compares local Lyapunov exponents to curvature, provides a solid
theoretical foundation for developing explicit numerical methods for stiff
problems.

Explicit ERK methods are ideally suited to
problems where the numerical stiffness comes from a linear term,
as they allow for relatively large step sizes.
Conventionally, ERK methods treat the linear term exactly.
However, in this work we show that in the case where the linear term
is a matrix, it is not necessary to treat the linear term exactly in
order to remove linear stiffness.

ERK methods treat the
linear term by calculating exponentials
and related functions of the linear coefficient $L$.
This is not a problem when $L$ is just a number or
if~$L$ is a diagonal matrix, but it is a computational burden when
$L$ is a general matrix (even if it is sparse).
The Schur decomposition of the general matrix can be used to transform
the linear coefficient to a triangular matrix.
The diagonal part of the new linear term is treated
exactly by the ERK method and the strictly triangular part is
treated explicitly, together with the nonlinear term.

Schur decomposition is particularly useful for embedded ERK
methods, because otherwise, every time that the step size is
adjusted, matrix functions would have to be recalculated.
With Schur decomposition, only functions of diagonal matrices
need to be recalculated at each time step.
Since the Schur decomposition algorithm only needs to be run
once at the very beginning, this greatly optimizes embedded
ERK methods and makes them a viable choice for high-performance
computing.

Lastly, we would like to remark that while we have defined
stiffness in general, exponential integrators can only
circumvent linear stiffness.
If there is stiffness associated with
the nonlinearity, one could (perhaps periodically) linearize
the equation around a certain state \cite{Hochbruck10}.
However, the matrix $L$ in the resulting linear part
would not in general be diagonal and we would have to perform a
Schur decomposition every time a linearization is performed. Another potential
improvement is to account for the off-diagonal terms of
the triangular matrix from the Schur decomposition by
using optimized algorithms for calculating functions of triangular
matrices.

\section*{Conflict of interest}
The authors have no competing interests to declare that are relevant to the content of this article.

\section*{Data availability}
All data generated or analyzed during this study is included in this published article.

\bibliographystyle{spmpsci}
\bibliography{refs}

\begin{thebibliography}{10}
\providecommand{\url}[1]{{#1}}
\providecommand{\urlprefix}{URL }
\expandafter\ifx\csname urlstyle\endcsname\relax
  \providecommand{\doi}[1]{DOI~\discretionary{}{}{}#1}\else
  \providecommand{\doi}{DOI~\discretionary{}{}{}\begingroup
  \urlstyle{rm}\Url}\fi

\bibitem{Bowman05}
Bowman, J.C.: Robust efficient routines to compute $\phi_n(x)$ for n=1 to 4.
\newblock \url{https://github.com/dealias/triad/blob/master/phi.h} (2005)

\bibitem{Cartwright99}
Cartwright, J.H.: Nonlinear stiffness, lyapunov exponents, and attractor
  dimension.
\newblock Physics Letters A \textbf{264}(4), 298--302 (1999)

\bibitem{Certaine60}
Certaine, J.: The solution of ordinary differential equations with large time
  constants.
\newblock Mathematical methods for digital computers \textbf{1}, 128--132
  (1960)

\bibitem{Cox02}
Cox, S., Matthews, P.: Exponential time differencing for stiff systems.
\newblock J. Comp. Phys. \textbf{176}, 430--455 (2002)

\bibitem{Curtiss52}
Curtiss, C.F., Hirschfelder, J.O.: Integration of stiff equations.
\newblock Proceedings of the National Academy of Sciences \textbf{38}(3),
  235--243 (1952)

\bibitem{Higham08}
Higham, N.J.: Functions of matrices: theory and computation.
\newblock SIAM (2008)

\bibitem{Hochbruck05}
Hochbruck, M., Ostermann, A.: Explicit exponential {R}unge--{K}utta methods for
  semilinear parabolic problems.
\newblock SIAM J. Numer. Anal. \textbf{43}, 1069--1090 (2005)

\bibitem{Hochbruck10}
Hochbruck, M., Ostermann, A.: Exponential integrators.
\newblock Acta Numerica \textbf{19}, 209--286 (2010)

\bibitem{Krogstad05}
Krogstad, S.: Generalized integrating factor methods for stiff pdes.
\newblock Journal of Computational Physics \textbf{203}(1), 72--88 (2005)

\bibitem{Lambert91}
Lambert, J.D.: Numerical methods for ordinary differential systems: the initial
  value problem.
\newblock John Wiley \& Sons, Inc. (1991)

\bibitem{Moler03}
Moler, C., Van~Loan, C.: Nineteen dubious ways to compute the exponential of a
  matrix, twenty-five years later.
\newblock SIAM review \textbf{45}(1), 3--49 (2003)

\bibitem{Prasolov94}
Prasolov, V.V.: Problems and theorems in linear algebra, vol. 134.
\newblock American Mathematical Soc. (1994)

\bibitem{Zoto23embed}
Zoto, T., Bowman, J.C.: Robust exponential {R}unge--{K}utta embedded pairs.
\newblock SIAM J. Sci. Comput.  (2023).
\newblock Submitted

\end{thebibliography}
\hypertarget{Bibliography}{}

\end{document}